\documentclass[11pt]{amsart}

\usepackage[colorlinks=true, pdfstartview=FitV, linkcolor=blue, citecolor=blue, urlcolor=blue, breaklinks=true]{hyperref}
\usepackage{
  amsmath,
  amsfonts,
  amssymb,
  amsthm,
  amscd,
  comment,
  paralist,
  etoolbox,
  mathtools,
  enumitem,
  mathdots
}
\usepackage[usenames,dvipsnames]{xcolor}
\usepackage[all]{xy}

\newcommand{\pushright}[1]{\ifmeasuring@#1\else\omit\hfill$\displaystyle#1$\fi\ignorespaces}
\newcommand{\pushleft}[1]{\ifmeasuring@#1\else\omit$\displaystyle#1$\hfill\fi\ignorespaces}
\makeatother

\newcommand\Z{\mathbb{Z}}

\newcommand\N{\mathbb{N}}
\newcommand\F{\mathbb{F}}

\newcommand\fa{\mathfrak{a}}

\newcommand\g{\mathfrak{g}}
\newcommand\fh{\mathfrak{h}}

\newcommand\fsl{\mathfrak{sl}}
\newcommand\ft{\mathfrak{t}}

\newcommand\cB{\mathcal{B}}

\newcommand\sm{\mathsf{m}}

\newcommand\ts{\textstyle}

\DeclareMathOperator{\ad}{ad}

\DeclareMathOperator{\ev}{ev}
\DeclareMathOperator{\Hom}{Hom}

\DeclareMathOperator{\id}{id}
\DeclareMathOperator{\Inf}{Inf}
\DeclareMathOperator{\maxSpec}{maxSpec}
\DeclareMathOperator{\Span}{Span}
\DeclareMathOperator{\Spec}{Spec}

\newtheorem{theo}{Theorem}[section]
\newtheorem{prop}[theo]{Proposition}
\newtheorem{lem}[theo]{Lemma}
\newtheorem{cor}[theo]{Corollary}

\theoremstyle{definition}

\newtheorem{rem}[theo]{Remark}
\newtheorem{eg}[theo]{Example}

\numberwithin{equation}{section}
\allowdisplaybreaks

\newtoggle{comments}
\newtoggle{details}
\newtoggle{prelimnote}
\newtoggle{detailsnote}

\toggletrue{detailsnote}   

\iftoggle{comments}{%
  \newcommand{\acomments}[1]{
    \ \\
    {\color{red}
      \textbf{AS:} #1
    }
    \ \\
    }
  \newcommand{\tcomments}[1]{
    \ \\
    {\color{red}
      \textbf{TM:} #1
    }
    \ \\
    }
}{%
  \newcommand{\acomments}[1]{}
  \newcommand{\tcomments}[1]{}
}

\iftoggle{details}{%
  \newcommand{\details}[1]{
      \ \\
      {\color{OliveGreen}
        \textbf{Details:} #1
      }
      \ \\
  }
}{%
  \newcommand{\details}[1]{}
}

\iftoggle{prelimnote}{%
  \newcommand{\prelim}{\textsc{Preliminary version} \bigskip}
}{%
  \newcommand{\prelim}{}
}

\begin{document}
%

\title{Invariant polynomials on truncated multicurrent algebras}

\author{Tiago Macedo}
\address[T.M.]{Department of Mathematics and Statistics, University of Ottawa and
Department of Science and Technology, Federal University of S\~ao Paulo, S\~ao Jos\'e dos Campos}
\urladdr{\url{http://www.ict.unifesp.br/tmacedo/}}
\email{tmacedo@unifesp.br}
\thanks{The first author was supported by CNPq grant 232462/2014-3.}

\author{Alistair Savage}
\address[A.S.]{
  Department of Mathematics and Statistics \\
  University of Ottawa
}
\urladdr{\href{http://alistairsavage.ca}{alistairsavage.ca}, \textrm{\textit{ORCiD}:} \href{https://orcid.org/0000-0002-2859-0239}{orcid.org/0000-0002-2859-0239}}
\email{alistair.savage@uottawa.ca}
\thanks{The second author was supported by Discovery Grant RGPIN-2017-03854 from the Natural Sciences and Engineering Research Council of Canada.}

\begin{abstract}
  We construct invariant polynomials on truncated multicurrent algebras, which are Lie algebras of the form $\mathfrak{g} \otimes_\mathbb{F} \mathbb{F}[t_1,\dotsc,t_\ell]/I$, where $\mathfrak{g}$ is a finite-dimensional Lie algebra over a field $\mathbb{F}$ of characteristic zero, and $I$ is a finite-codimensional ideal of $\mathbb{F}[t_1,\dotsc,t_\ell]$ generated by monomials.  In particular, when $\mathfrak{g}$ is semisimple and $\mathbb{F}$ is algebraically closed, we construct a set of algebraically independent generators for the algebra of invariant polynomials.  In addition, we describe a transversal slice to the space of regular orbits in $\mathfrak{g} \otimes_\mathbb{F} \mathbb{F}[t_1,\dotsc,t_\ell]/I$.  As an application of our main result, we show that the center of the universal enveloping algebra of $\mathfrak{g} \otimes_\mathbb{F} \mathbb{F}[t_1,\dotsc,t_\ell]/I$ acts trivially on all irreducible finite-dimensional representations provided $I$ has codimension at least two.
\end{abstract}

\subjclass[2010]{Primary 17B05, 17B08; Secondary 17B35, 17B70}
\keywords{Truncated current algebra, Takiff algebra, equivariant map algebra, invariant polynomial, universal enveloping algebra, transversal slice, regular orbit}

\prelim

\maketitle
\thispagestyle{empty}

\tableofcontents

%
\section{Introduction}
%

\subsection{Motivation} \label{subsec:motivation}

Suppose $A$ is a finitely-generated commutative associative unital algebra over a field $\F$ of characteristic zero.  All tensor products will be over $\F$.  Furthermore, suppose that $\Gamma$ is a finite group acting on $A$ and on a finite-dimensional Lie algebra $\g$ by automorphisms.  The corresponding \emph{equivariant map algebra} $(\g \otimes A)^\Gamma$ can be viewed as the Lie algebra of $\Gamma$-equivariant algebraic maps from $\Spec A$ to $\g$.  Equivariant map algebras are a large class of Lie algebras generalizing loop algebras and current algebras, which are vital to the theory of affine Lie algebras, and are an extremely active area of research.  We refer the reader to the survey \cite{NS13} for an overview of the field.

In the case that $\Gamma$ is abelian and acts freely on $\maxSpec A$, $\F$ is algebraically closed, and $\g$ is semisimple, one can use the twisting and untwisting functors defined in \cite{FKKS12} to reduce the study of finite-dimensional representations of $(\g \otimes A)^\Gamma$ to the study of representations of the corresponding \emph{untwisted} map algebra $\g \otimes A$, where $\Gamma$ is trivial (see \cite[Th.~2.10]{FKKS12}).

Now assume that $\Gamma$ is trivial and $\g$ is semisimple.  Since $A$ is finitely generated, we may assume that $A = \F[t_1,\dotsc,t_\ell]/I$ for some ideal $I$ of $\F[t_1,\dotsc,t_\ell]$.  In particular, maximal ideals of $A$ correspond to maximal ideals of $\F[t_1,\dotsc,t_\ell]$ containing $I$.  It is known that any finite-dimensional $\g \otimes A$-module is a tensor product of modules supported at single points, that is, modules annihilated by $\g \otimes (\sm^N + I)$ for some maximal ideal $\sm$ of $A$ containing $I$, and some $N \in \N$.  (This follows, for example, from \cite[Prop.~2.4]{FKKS12}.)  Suppose $V$ is a $\g \otimes A$-module annihilated by $\g \otimes (\sm^N + I)$. Translating if necessary, we may assume that $\sm = (t_1,\dotsc,t_\ell)$ is the maximal ideal corresponding to the origin.  Since $\sm^N + I$ annihilates $V$, so does $\sm^N$.  Hence, $V$ is naturally a module for the quotient $\g \otimes \F[t_1,\dotsc,t_\ell]/\sm^N$.  It follows that we can reduce the study of finite-dimensional $\g \otimes A$-modules to the case $A = \F[t_1,\dotsc,t_\ell]/I$, where $I$ is an ideal of $\F[t_1,\dotsc,t_\ell]$ of finite codimension, generated by monomials.  We call these Lie algebras \emph{truncated multicurrent algebras}.

In the case $\ell=1$, the Lie algebras $\g \otimes \F[t]/(t^{n+1})$, called \emph{truncated current algebras}, or \emph{generalized Takiff algebras}, have appeared in many places in the literature.  Most relevant to the current paper is the work of Takiff, who considered invariant polynomials in the case $n=1$ in \cite{Tak71}.  Ra\"is and Tauvel considered the case of arbitrary $n$ in \cite{RT92}, as did Geoffriau in \cite{Geo94a,Geo95}.  More recently, graded modules for Takiff algebras were investigated in \cite{CG09}, highest-weight theory for truncated current algebras were considered by Wilson in \cite{Wil11}, and connections to the geometric Langlands program were studied by Kamgarpour in \cite{Kam16}.

Motivated by the above discussion, in the current paper we study the structure of truncated multicurrent algebras.  A particularly useful tool in the representation theory of Lie algebras has been the action of the center of the universal enveloping algebra.  On the other hand, the Duflo isomorphism is an algebra isomorphism between the center of the universal enveloping algebra of a finite-dimensional Lie algebra and the invariants in its symmetric algebra (see \cite{Duf77}).  The goal of the current paper is to describe this space of invariants for truncated multicurrent algebras.

\subsection{Main results}

Suppose $\F$ is a field of characteristic zero, $\g$ is a finite-dimensional Lie algebra, and $\ell$ is a positive integer.  Define a partial order on the set $\Omega = \N^\ell$ by
\[
  (n_1,\dotsc,n_\ell) \le (m_1,\dotsc,m_\ell) \iff n_i \le m_i \quad \text{for all } i \in \{1,\dotsc,\ell\}.
\]
Suppose $\Omega_0$ is a subset of $\Omega$ that is invariant under the action of $\Omega$ on itself by componentwise addition, and such that $\Omega_1 = \Omega \setminus \Omega_0$ is finite.  Then the span $\F \langle t^\omega \mid \omega \in \Omega_0 \rangle$ is an ideal of $\F[t_1,\dotsc,t_\ell]$ and we define
\[
  A = \F[t_1,\dotsc,t_\ell] / \F \langle t^\omega \mid \omega \in \Omega_0 \rangle.
\]

In Section~\ref{subsec:poly-defs}, we associate to any polynomial $p \in S(\g)$ a family $p_\omega$, $\omega \in \Omega$, of elements of $S(\g \otimes A)$.  Suppose that $p \in S(\g)^\g$, that is, $p$ is invariant under the action of $\g$ on $S(\g)$ induced by the adjoint action.  We then show, in Proposition~\ref{prop:pgamma-invariant}, that $p_\omega \in S(\g \otimes A)^{\g \otimes A}$ for certain values of $\omega$.  In particular, if $\Omega_1$ has a greatest element $\mu$, which we will assume for the remainder of this introduction, then, for all $k \in \N$,
\[
  p \in S^k(\g)^\g \implies p_\omega \in S^k(\g \otimes A)^{\g \otimes A} \text{ for } \omega \in k \mu - \Omega_1.
\]
(Note that $\Omega_1$ has a greatest element if and only if $A$ is a symmetric algebra.  If the greatest element is $\mu$, the corresponding trace map is projection onto $A_\mu$.)  In other words, given invariant polynomials in $S(\g)$, one can construct invariant polynomials in $S(\g \otimes A)$.  Furthermore, in Proposition~\ref{prop:poly-alg-ind}, we show that a collection of polynomials
\[
  p^{(j)} \in S^{k_j}(\g),\quad 1 \le j \le r,
\]
is algebraically independent if and only if the associated collection
\[
  p^{(j)}_\omega \in S^{k_j}(\g \otimes A),\quad 1 \le j \le r,\ \omega \in k_j \mu - \Omega_1,
\]
is algebraically independent.

Now specialize to the case where $\F$ is algebraically closed and $\g$ is semisimple.  In this case, we may identify $\g$ with $\g^*$ via the Killing form, and it is known that $S(\g^*)^\g$ is isomorphic to a finitely-generated polynomial algebra.  Thus, we may choose algebraically independent generators $p^{(j)} \in S^{k_j}(\g^*)^\g$, $j \in \{1,\dotsc,r\}$, of $S(\g^*)^\g$.  Our main result, Theorem~\ref{theo:inv-polys}, states that the corresponding polynomials $p^{(j)}_\omega$, $\omega \in \Omega_1$, $j \in \{1,\dotsc,r\}$, form a system of algebraically independent generators of $S\big( (\g \otimes A)^* \big)^{\g \otimes A}$.  In addition, we describe a transversal slice $\ft$ to the regular orbits in $\g \otimes A$ under the adjoint action, such that restriction to $\ft$ induces an isomorphism of algebras from $S\big( (\g \otimes A)^* \big)^{\g \otimes A}$ to $S(\ft^*) = \F[\ft]$.  These results generalize \cite[Th.~4.5]{RT92}, which considers the truncated current (i.e.\ $\ell=1$) case.  In fact, we also generalize, in Proposition~\ref{prop:chi-regular}, two other results from that paper, by computing the index of the Lie algebra $\g \otimes A$ and characterizing the regular linear forms on $\g \otimes A$.  (Note that Proposition~\ref{prop:chi-regular} does not require the assumption that $\F$ is algebraically closed nor that $\g$ is semisimple.)

Finally, as an application of our main result, we show, in Theorem~\ref{theo:center-acts-trivially}, that the center of the universal enveloping algebra of $\g \otimes A$ acts trivially (i.e.\ via the augmentation map) on all irreducible finite-dimensional representations (but not necessarily on all finite-dimensional representations), provided $\mu > 0$.  It follows that action of this center does \emph{not} separate blocks in the category of finite-dimensional representations.

\subsection{Future directions} \label{subsec:future-directions}

The results of the current paper suggest many natural possible directions of future research.  We list some of these here.

\begin{asparaenum}
  \item Kostant proved that, if $\g$ is a reductive Lie algebra, then $S(\g^*)$ is free over $S(\g^*)^\g$ (see \cite[Th.~11]{Kos63}), and hence that the universal enveloping algebra $U(\g)$ is free over its center $Z(\g)$ (see \cite[Th.~21]{Kos63}).  In the case that $\g$ is semisimple and $A = \F[t]/(t^2)$, analogous results were proved in \cite[Th.~3.3]{Geo94a}.  This was generalized to the case $A = \F[t]/(t^{n+1})$, $n \in \N$, in \cite[Th.~A.4]{Mus01}.  It is thus natural to ask if one can prove these results more generally in the case where $\g \otimes A$ is a truncated multicurrent algebra.

  \item If $\g$ is semisimple, one has the Harish-Chandra isomorphism between $Z(\g)$ and $S(\fh)^W$, the algebra of polynomials in a Cartan subalgebra $\fh$ of $\g$, invariant under the action of the Weyl group $W$ (see, for example, \cite[\S23.3]{Hum78}).  An analogue of this result in the case $A=\F[t]/(t^{n+1})$, $n \in \N$, was proved in \cite[Th.~4.7]{Geo95}.  It would be interesting to construct an analogue of the Harish-Chandra homomorphism in the more general setting of truncated multicurrent algebras.

  \item In Section~\ref{subsec:motivation}, we recalled how, when the group $\Gamma$ acts freely on $\maxSpec A$, one can reduce the study of finite-dimensional representations of $(\g \otimes A)^\Gamma$ to the untwisted setting in which $\Gamma$ is the trivial group.  However, this procedure fails when $\Gamma$ does not act freely.  (We refer the reader to \cite{NSS12} for a discussion of this more general setup.)  It is an open problem to examine the structure of the invariant polynomials $S \big( (\g \otimes A)^\Gamma \big)^{(\g \otimes A)^\Gamma}$ in general.

  \item The theory of equivariant map algebras was extended to the super case in \cite{Sav14,CMS16,Bag15,CM16}.  It would be interesting to also extend the results of the current paper to that setting.
\end{asparaenum}

\iftoggle{detailsnote}{
\subsection*{Note on the arXiv version} For the interested reader, the tex file of the \href{https://arxiv.org/abs/1607.06411}{arXiv version} of this paper includes hidden details of some straightforward computations and arguments that are omitted in the pdf file.  These details can be displayed by switching the \texttt{details} toggle to true in the tex file and recompiling.
}{}

\subsection*{Acknowledgment}

The authors would like to thank V.~Chari for helpful discussions.

%
\section{Preliminaries}
%

We fix a field $\F$ of characteristic zero, and all algebras and tensor products are over $\F$.  We assume that all associative algebras possess multiplicative units.  If $U$ is a subspace of a vector space $V$, then, for $v_1,v_2 \in V$, we write $v_1 \equiv v_2 \mod U$ to indicate that $v_1 - v_2 \in U$ (in other words, the images of $v_1$ and $v_2$ in $V/U$ are equal.)  We let $\N = \{0,1,2,\dotsc\}$ denote the set of nonnegative integers and $\N_+ = \{1,2,3,\dotsc\}$ denote the set of positive integers.

\subsection{Polynomial functions and the adjoint action} \label{subsec:poly-funcs-adjoint}

For a vector space $V$, we let $S(V) = \bigoplus_{k \in \N} S^k(V)$ denote the symmetric algebra on $V$ with its usual grading by degree.  Let $\fa$ be an arbitrary Lie algebra.  Recall that $\fa$ acts on itself via the adjoint action:
\begin{equation} \label{eq:adjoint-action-g}
  x \cdot y = [x,y],\quad \text{for all } x,y \in \fa.
\end{equation}
We then have an action of $\fa$ on $S(\fa)$ defined inductively by \eqref{eq:adjoint-action-g} and
\[
  x \cdot (p_1p_2) = (x \cdot p_1) p_2 + p_1 (x \cdot p_2),\quad \text{for all } x \in \fa,\ p_1,p_2 \in S(\fa).
\]
Note that this implies that $\fa$ acts as zero on $S^0(\fa) = \F$.  We let
\[
  S(\fa)^\fa = \{p \in S(\fa) \mid x \cdot p = 0 \text{ for all } x \in \fa\}
\]
denote the set of $\fa$-invariant elements of $S(\fa)$.

The corresponding dual action of $\fa$ on $\fa^* = \Hom_\F (\fa,\F)$ is defined by
\[
  (x \cdot f)(y)
  = -f(x \cdot y)
  = f([y,x]),\quad x,y \in \fa,\ f \in \fa^*.
\]
Since elements of $\fa^*$ are $\F$-linear maps $\fa \to \F$, they induce algebra homomorphisms $S(\fa) \to \F$.  (We use here the universal property of the symmetric algebra and the fact that $\F$ is commutative.)  Thus, we have a map
\[
  S(\fa) \times \fa^* \to \F,\quad (p,f) \mapsto p(f),\quad p \in S(\fa),\ f \in \fa^*.
\]
In this way, we can view elements of $S(\fa)$ as polynomial functions on $\fa^*$.

More generally, fix a commutative associative algebra $B$.  (We shall be particularly interested in the cases $B = \F[\Gamma]$, $B = \F[s]$, and $B = \F[\Gamma] \otimes \F[s]$, where $\Gamma$ is an abelian group and $s$ is an indeterminate.)  Then we have an $\F$-linear map determined by
\[
  B \otimes \fa^* \otimes \fa \to B,\quad b \otimes f \otimes x \mapsto f(x) b,
  \quad \text{for all } b \in B,\ f \in \fa^*,\ x \in \fa.
\]
As above, any element of $B \otimes \fa^*$ thus induces an algebra homomorphism $S(\fa) \to B$.  So
we have a map
\[
  S(\fa) \times \big( B \otimes \fa^*\big) \to B.
\]
We can therefore also view elements of $S(\fa)$ as polynomial functions on $B \otimes \fa^*$, taking values in $B$. More generally, we view elements of $S(\fa) \otimes B$ in the same way via
\[
  (p \otimes b)(f) = p(f)b,\quad
  \text{ for all } p \in S(\fa),\ f \in \fa^* \otimes B,\ b \in B.
\]

We define the structure of an $\fa \otimes B$-module on $\fa^* \otimes B$ by
\begin{equation} \label{eq:adjoint-dual-action-coeffs}
  (x \otimes b_1) \cdot (f \otimes b_2) = (x \cdot f) \otimes b_1 b_2,\quad
  x \in \fa,\ f \in \fa^*,\ b_1,b_2 \in B,
\end{equation}
and extending by linearity.  We also define the structure of an $\fa \otimes B$-module on $S(\fa) \otimes B$ by
\begin{equation} \label{eq:adjoint-poly-action-coeffs}
  (x \otimes b_1) \cdot (p \otimes b_1) = (x \cdot p) \otimes b_1 b_2,\quad
  x \in \fa,\ p \in S(\fa),\ b_1,b_2 \in B,
\end{equation}
and extending by linearity.  It follows immediately from this definition, that if $p \in S(\fa)^\fa$, then $(\fa \otimes B) \cdot (p \otimes b) = 0$ for all $b \in B$.

\begin{lem} \label{lem:linearization}
  Let $s$ be an indeterminate.  For all $P \in S(\fa) \otimes B$, $F \in \fa^* \otimes B$, and $X \in \fa \otimes B$, we have
  \begin{equation} \label{eq:linearization}
    P(F + s (X \cdot F)) \equiv P(F) - s (X \cdot P)(F) \mod s^2 B[s].
  \end{equation}
\end{lem}

\begin{proof}
  Since both sides of \eqref{eq:linearization} are $\F$-linear in $P$, it suffices to prove the result for $P \in S^k(\fa) \otimes B$, $k \in \N$.  Similarly, from the definition \eqref{eq:adjoint-poly-action-coeffs}, it suffices to consider $P \in S^k (\fa) \otimes \F \cong S^k (\fa)$.  In addition, it is enough to consider $X$ of the form $x \otimes b_1$ and $F$ of the form $f \otimes b_2$, where $x \in \fa$, $f \in \fa^*$, and $b_1, b_2 \in B$.  We use induction on $k$.  The $k=0$ case is trivial.

  For $P = y \in \fa = S^1(\fa)$, we have
  \begin{align*}
    y \big( f \otimes b_2 + s (x \otimes b_1) \cdot (f \otimes b_2) \big)
    &= y \big( f \otimes b_2 + s (x \cdot f) \otimes b_1 b_2 \big) \\
    &= f(y) b_2 + s(x \cdot f)(y) b_1 b_2 \\
    &= f(y) b_2 - s f(x \cdot y) b_1 b_2 \\
    &= y(f \otimes b_2) - s \big( (x \otimes b_1) \cdot y \big) (f \otimes b_2).
  \end{align*}
  Now suppose that, for some $N \in \N_+$, the result holds for $P$ in $S^k(\fa)$, $k \le N$.  Let $X \in \fa \otimes B$, $F \in \fa^* \otimes B$, $P \in S^k(\fa)$, and $Q \in S^\ell(\fa)$, with $k,\ell \le N$.  Then we have
  \begin{align*}
    (PQ)(F + s(X \cdot F))
    &= P(F + s(X \cdot F)) Q(F + s(X \cdot F)) \\
    &\equiv \big( P(F) - s(X \cdot P)(F) \big) \big( Q(F) - s (X \cdot Q)(F) \big) \mod s^2 B[s] \\
    &\equiv P(F) Q(F) - s \big( (X \cdot P)(F) Q(F) + P(F) (X \cdot Q)(F) \big) \mod s^2 B[s] \\
    &\equiv (PQ)(F) - s \big( X \cdot (PQ) \big)(F) \mod s^2 B[s].
  \end{align*}
  This completes the proof of the inductive step.
\end{proof}

\subsection{Graded associative algebras}

Fix $\ell \in \N$, and define
\[
  \Gamma = \Z^\ell,\quad \Omega = \N^\ell.
\]
So $\Gamma$ is an abelian group and $\Omega$ is an abelian monoid under the usual addition.  There are natural actions of $\Omega$ and $\Gamma$ on themselves via their operations.

We define a partial order on $\Gamma$ by
\[
  (n_1,\dotsc,n_\ell) \le (m_1,\dotsc,m_\ell) \iff n_i \le m_i \quad \text{for all } i \in \{1,\dotsc,\ell\}.
\]
We write $(n_1,\dotsc,n_\ell) < (m_1,\dotsc,m_\ell)$ if $(n_1,\dotsc,n_\ell) \le (m_1,\dotsc,m_\ell)$ and $(n_1,\dotsc,n_\ell) \ne (m_1,\dotsc,m_\ell)$.

Suppose $B = \bigoplus_{\gamma \in \Gamma} B_\gamma$ is a finite-dimensional $\Gamma$-graded commutative $\F$-algebra.  For $\nu \in \Gamma$, we define
\[
  B_{> \nu} = \sum_{\nu < \gamma} B_\gamma.
\]
Let
\[
  B^* = \bigoplus_{\gamma \in \Gamma} B_\gamma^*,\quad B_\gamma^* = \Hom_\F(B_\gamma,\F),
\]
denote the restricted dual space.  This is naturally a $\Gamma$-graded vector space with the elements of $B_\gamma^*$ being of degree $- \gamma$.  In other words, $(B^*)_{-\gamma} = B_\gamma^*$.  The space $B^*$ is also naturally a $B$-module via the action defined by
\[
  (b \cdot f)(b') = f(bb'),\quad b,b' \in B,\ f \in B^*.
\]
This action is compatible with the $\Gamma$-gradings:
\[
  B_{\gamma_1} (B^*)_{\gamma_2} = (B^*)_{\gamma_1 + \gamma_2}.
\]

If $\mathcal{B}$ is some subset of $B$, we write
\[
  \F \langle \mathcal{B} \rangle = \F \langle b \mid b \in \cB \rangle = \Span_\F \mathcal{B}
\]
to denote the $\F$-span of the set $\mathcal{B}$.

Consider the polynomial algebra $\F[t_1,\dotsc,t_\ell]$ in $\ell$ indeterminates.  For $\omega = (\omega_1,\dotsc,\omega_\ell) \in \Omega$, we define $t^\omega = t_1^{\omega_1} t_2^{\omega_2} \dotsm t_\ell^{\omega_\ell}$.  Fix a subset $\Omega_0 \subseteq \Omega$ that is stable under the action of $\Omega$ (i.e.\ such that $\omega + \gamma \in \Omega_0$ for all $\omega \in \Omega$ and $\gamma \in \Omega_0$), and such that the complement $\Omega_1 = \Omega \setminus \Omega_0$ is finite.  Thus, $I = \F \langle t^\omega \mid \omega \in \Omega_0 \rangle$ is an ideal of $\F[t_1,\dotsc,t_\ell]$, and we define the $\F$-algebra
\[
  A = \F[t_1,\dotsc,t_\ell]/I.
\]
Let $\tau^\omega$ denote the image of $t^\omega$ in $A$, for $\omega \in \Omega$.  So $A$ has a basis
\[
  \{\tau^\omega \mid \omega \in \Omega_1\}.
\]
Let
\[
  \{\epsilon^\omega \mid \omega \in \Omega_1\}
\]
denote the dual basis of $A^*$.  The algebra $A$ is naturally $\Gamma$-graded, with the element $\tau^\omega$ being of degree $\omega$.  Similarly, $A^*$ is a $\Gamma$-graded vector space, with $\epsilon^\omega$ of degree $-\omega$.

Consider the group algebra
\[
  \F[\Gamma] = \{q^\gamma \mid \gamma \in \Gamma\},
\]
written in exponential notation, where $q$ is a formal symbol.  Define an $\F$-linear map
\begin{equation} \label{eq:imath-def}
  \imath \colon A \hookrightarrow \F[\Gamma],\quad \tau^\omega \mapsto q^\omega,\quad \omega \in \Omega_1.
\end{equation}
This map is manifestly injective.  We will continue to use the notation $\imath$ to denote the tensor product of $\imath$ with identity maps.  For example, if $\fa$ is a Lie algebra and $B$ is an associative algebra, we have the injective $\F$-linear map
\[
  \imath = \id_\fa \otimes \imath \otimes \id_B \colon \fa \otimes A \otimes B \hookrightarrow \fa \otimes \F[\Gamma] \otimes B.
\]
We also define the $\F$-linear map (which is manifestly injective)
\begin{equation} \label{eq:jmath-def}
  \jmath \colon A^* \hookrightarrow \F[\Gamma],\quad \epsilon^\omega \mapsto q^{-\omega},\quad \omega \in \Omega_1,
\end{equation}
and use the same notation $\jmath$ to denote the tensor product with identity maps.  Note that the maps \eqref{eq:imath-def} and \eqref{eq:jmath-def} are both maps that send homogeneous elements to their degrees.

\begin{rem} \label{rem:A-max-element-case}
  In the sequel, we shall be especially interested in the case where $\Omega_1$ has a greatest element $\mu$.  In this case, we have
  \begin{gather*}
    \Omega_1 = \{\omega \in \Omega \mid \omega \le \mu\},\quad
    \Omega_0 = \{\omega \in \Omega \mid \omega \not \le \mu\},\quad \text{and} \\
    A \cong \F[t_1]/(t_1^{\mu_1+1}) \otimes \F[t_2]/(t_2^{\mu_2+1}) \otimes \dotsb \otimes \F[t_\ell]/(t_\ell^{\mu_\ell+1}).
  \end{gather*}
\end{rem}

%
\section{Invariant polynomials}
%

\subsection{Definitions} \label{subsec:poly-defs}

Fix a finite-dimensional Lie algebra $\g$.  Through the paper, we will use the symbols $X,Y,Z$ to denote elements of $\g \otimes A$ or $\g \otimes \F[\Gamma]$ and the symbols $F,G,H$ to denote elements of $\g^* \otimes A^*$ or $\g^* \otimes \F[\Gamma]$.  We use the corresponding lowercase letters to denote the components of these elements:
\[
  X = \sum_{\omega \in \Omega_1} x_\omega \otimes \tau^\omega,\quad
  F = \sum_{\omega \in \Omega_1} f_\omega \otimes \epsilon^\omega,\quad \text{etc.}
\]

Let
\begin{equation} \label{eq:a-def}
  a = \sum_{\omega \in \Omega_1} \tau^\omega \in A.
\end{equation}
We have the associated $\F$-linear map
\begin{equation}
  \tau_a \colon \g \to \g \otimes A,\quad x \mapsto x \otimes a.
\end{equation}
This induces an algebra homomorphism
\begin{equation}
  \tau_a \colon S(\g) \to S(\g \otimes A),
\end{equation}
which we denote by the same symbol.  Now, the algebra $S(\g \otimes A)$ is naturally $\Gamma$-graded (as an algebra).  For $p \in S(\g)$ and $\gamma \in \Gamma$, we let $p_\gamma$ denote the degree $\gamma$ component of $\tau_a(p)$.  Thus, we have
\begin{equation}
  \tau_a(p) = \sum_{\gamma \in \Gamma} p_\gamma,\quad p_\gamma \in S(\g \otimes A)_\gamma.
\end{equation}

For subsets $\Upsilon_1,\dotsc,\Upsilon_k \subseteq \Gamma$, define
\begin{gather*}
  \Upsilon_1 + \Upsilon_2 + \dotsb + \Upsilon_k := \{\gamma_1 + \gamma_2 + \dotsb + \gamma_k \mid \gamma_i \in \Upsilon_i,\ i \in \{1,\dotsc,k\}\} \subseteq \Gamma.
\end{gather*}
Also, for a subset $\Upsilon \subseteq \Gamma$ and $k \in \N_+$, we define
\[
  \Upsilon^k = \Upsilon + \dotsb + \Upsilon \quad \text{($k$ factors)},\quad
  \Upsilon^0 = \{0\},
\]
and
\[
  - \Upsilon = \{- \gamma \mid \gamma \in \Upsilon\} \subseteq \Gamma.
\]

\begin{rem} \label{rem:zero-p-gamma}
  Note that, for $p \in S^k(\g)$, we have $p_\gamma = 0$ for $\gamma \not \in \Omega_1^k$.
\end{rem}

Since $\g$ and $A$ are finite dimensional, we may naturally identify $(\g \otimes A)^*$ with $\g^* \otimes A^*$.  Recall the definition of the map $\jmath$ given in \eqref{eq:jmath-def}.

\begin{lem} \label{lem:p-jmath-expansion}
  For $p \in S(\g)$ and $F \in \g^* \otimes A^*$, we have
  \begin{equation} \label{eq:p-jmath-F-decomp}
    p(\jmath(F)) = \sum_{\gamma \in \Omega} p_\gamma(F) q^{-\gamma}.
  \end{equation}
\end{lem}

\begin{proof}
  Since both sides of \eqref{eq:p-jmath-F-decomp} are $\F$-linear in $p$, it suffices to consider $p \in S^k(\g)$ for some $k \in \N_+$.  Then, for $F = \sum_{\omega \in \Omega_1} f_\omega \otimes \epsilon^\omega$, we have
  \[
    p (\jmath (F))
    = p \left( \sum_{\omega \in \Omega_1} f_\omega \otimes q^{-\omega} \right)
    = \sum_{\omega \in \Omega_1} p(f_\omega) \otimes q^{-k\omega}.
  \]
  Thus, it suffices to consider the case $F = f \otimes \epsilon^\omega$ for some $f \in \g^*$, and $\omega \in \Omega_1$.

  Since $p_\gamma(F) = 0$ for $\gamma \ne k \omega$, we have
  \[
    \tau_a(p)(F) = p_{k \omega}(F) = p(f)
    \quad \text{and} \quad
    p(\jmath(F)) = p(f \otimes q^{- \omega}) = p(f) q^{-k\omega}.
  \]
  The result follows.
\end{proof}

Notice that, in the setup of Lemma~\ref{lem:p-jmath-expansion}, we have that $p(\jmath(F))$ is an element of $\F[\Omega^{-1}]$, which is simply a polynomial algebra in $q_i := q^{-1_i}$, $i \in \{1,\dotsc,\ell\}$, where $1_i \in \Omega = \N^\ell$ is the element with $i$-th component equal to one and all other components equal to zero.  We can thus apply the partial derivatives $\partial_i := \frac{\partial}{\partial q_i}$.

\begin{cor} \label{cor:p-gamma-deriv-expression}
  For $p \in S(\g)$ and $\gamma \in \Omega$, we have
  \begin{equation}
    p_\gamma = \frac{1}{\gamma_1! \gamma_2! \dotsm \gamma_\ell!} \partial_1^{\gamma_1} \partial_2^{\gamma_2} \dotsb \partial_\ell^{\gamma_\ell} \big|_0 (p \circ \jmath),
  \end{equation}
  where the notation $\partial_1^{\gamma_1} \partial_2^{\gamma_2} \dotsb \partial_\ell^{\gamma_\ell} \big|_0$ denotes evaluation of the iterated derivative at $q_1=q_2=\dotsb=q_\ell=0$.
\end{cor}

\begin{proof}
  This follows immediately from Lemma~\ref{lem:p-jmath-expansion}.
\end{proof}

\begin{rem} \label{rem:pgamma-RT}
  Suppose $\ell=1$, so that $A \cong \F[t]/(t^{m+1})$ for some $m$.  Then it follows from Lemma~\ref{lem:p-jmath-expansion} that, for $p \in S^k(\g)$, and $i \in \N = \Omega$, the polynomial $p_{km-i} \in S(\g \otimes A)_{km-i}$ is equal to $P_i$ as defined in \cite[\S3.1]{RT92}.
  \details{
    Consider the first centered equation of \cite[\S3.1]{RT92}, which defines the $P_i$.  We have
    \[
      p \left( \sum_{n=0}^m q^{n-m} f_{m-n} \right)
      = q^{-km} p \left( \sum_{n=0}^m q^n f_{m-n} \right)
      = q^{-km} \sum_{i \ge 0} q^i P_i \left( \sum_{j=0}^m f_j \otimes \epsilon^j \right).
    \]
    Then the result follows from Lemma~\ref{lem:p-jmath-expansion}.
  }
\end{rem}

The following lemma will be useful in the proof of Proposition~\ref{prop:poly-alg-ind}.

\begin{lem} \label{lem:double-invariant-poly}
  Assume $\Omega_1$ has a greatest element $\mu$.  Suppose $p \in S(\g)$ and $\gamma = (\gamma_1,\dotsc,\gamma_\ell) \in \Omega$.  By Remark~\ref{rem:A-max-element-case}, we have
  \[
    A \cong B \otimes \F[t_\ell]/(t_\ell^{\mu_\ell}),\quad \text{where} \quad
    B = \F[t_1,\dotsc,t_{\ell-1}]/ \F \langle t^\omega \mid \omega \not \le (\mu_1,\dotsc,\mu_{\ell-1}) \rangle.
  \]
  Identifying $A$ with $B \otimes \F[t_\ell]/(t_\ell^{\mu_\ell})$ via this isomorphism, we can form $p_{(\gamma_1,\dotsc,\gamma_{\ell-1})} \in S(\g \otimes B)$ and $\left( p_{(\gamma_1,\dotsc,\gamma_{\ell-1})} \right)_{\gamma_\ell} \in S(\g \otimes A)$.  Then we have $\left( p_{(\gamma_1,\dotsc,\gamma_{\ell-1})} \right)_{\gamma_\ell} = p_\gamma$.
\end{lem}

\begin{proof}
  Define
  \[ \ts
    b = \pi_B \left( \sum_{\omega \in \N^{\ell-1},\, \omega \le (\mu_1,\dotsc,\mu_{\ell-1})} t^\omega \right) \in B
    \quad \text{and} \quad
    c = \pi_\ell \left( \sum_{\omega_\ell \in \N,\, \omega_\ell \le \mu_\ell} t_\ell^{\omega_\ell} \right) \in \F[t_\ell]/(t_\ell^{\mu_\ell+1}),
  \]
  where
  \[
    \pi_B \colon \F[t_1,\dotsc,t_{\ell-1}] \twoheadrightarrow B
    \quad \text{and} \quad
    \pi_\ell \colon \F[t_\ell] \twoheadrightarrow \F[t_\ell]/(t_\ell^{\mu_\ell})
  \]
  are the canonical projections.  We then have the $\F$-linear maps
  \[
    \tau_b \colon \g \to \g \otimes B,\ x \mapsto x \otimes b \quad \text{and} \quad
    \tau_c \colon \g \otimes B \to \g \otimes A,\ x \mapsto x \otimes c.
  \]
  which induce algebra homomorphisms $\tau_b \colon S(\g) \to S(\g \otimes B)$ and $\tau_c \colon S(\g \otimes B) \to S(\g \otimes A)$.  Then the result follows from the fact that $\tau_c \circ \tau_b = \tau_a$, where $a$ is defined as in \eqref{eq:a-def}.
\end{proof}

\subsection{Invariance}

\begin{lem} \label{lem:adjoint-equivariance}
  For $X \in \g \otimes A$ and $F \in \g^* \otimes A^*$, we have
  \[
    \jmath (X \cdot F) - \imath(X) \cdot \jmath(F) \in \g^* \otimes \F \langle (\Omega_1 - \Omega_1) \setminus - \Omega_1 \rangle.
  \]
\end{lem}

\begin{proof}
  Since the expression $\jmath(X \cdot F) - \imath(X) \cdot \jmath(F)$ is $\F$-linear in both $X$ and $F$, it suffices to consider the case $X = x \otimes \tau^\omega$ and $F = f \otimes \epsilon^\lambda$ for some $x \in \g$, $f \in \g^*$, and $\omega, \lambda \in \Omega_1$.  We split the proof into two cases.  First, suppose that $\lambda - \omega \in \Omega_1$.  Then, for $\gamma \in \Omega_1$, we have
  \[
    \tau^\omega \cdot \epsilon^\lambda(\tau^\gamma)
    = \epsilon^\lambda(\tau^{\omega + \gamma})
    = \delta_{\lambda, \omega + \gamma}
    = \delta_{\lambda - \omega, \gamma}
    = \epsilon^{\lambda - \omega}(\tau^\gamma).
  \]
  and so
  \[
    X \cdot F = (x \cdot f) \otimes \epsilon^{\lambda-\omega}.
  \]
  Therefore
  \[
    \jmath (X \cdot F)
    = (x \cdot f) \otimes q^{\omega-\lambda}
    = \imath(X) \cdot \jmath(F).
  \]
  On the other hand, if $\lambda - \omega \not \in \Omega_1$, then $X \cdot F = 0$ and
  \[
    \imath(X) \cdot \jmath(F)
    = (x \otimes q^\omega) \cdot (f \otimes q^{-\lambda})
    = (x \cdot f) \otimes q^{\omega-\lambda}
    \in \g^* \otimes \F \langle (\Omega_1 - \Omega_1) \setminus - \Omega_1 \rangle. \qedhere
  \]
\end{proof}

\begin{lem} \label{lem:polynomial-error-terms}
  Suppose $p \in S^k(\g)$ for some $k \in \N_+$.  Then, for all $F \in \g^* \otimes \F \langle - \Omega_1 \rangle$ and $G \in \g^* \otimes \F \langle (\Omega_1 - \Omega_1) \setminus -\Omega_1 \rangle$, we have
  \[
    p(F+G) - p(F) \in \F \langle \Phi_k \rangle,
  \]
  where
  \begin{equation} \label{eq:Upsilonk-def}
    \Phi_k = \bigcup_{j=1}^k \Big( \big( (\Omega_1 - \Omega_1) \setminus -\Omega_1 \big)^j + (-\Omega_1)^{k-j} \Big).
  \end{equation}
\end{lem}

\begin{proof}
  It suffices to consider the case where $p$ is a monomial, since monomials span $S^k(\g)$.  Therefore, take $p=x_1 x_2 \dotsm x_k$, for some $x_1,\dotsc,x_k \in \g$.  Then we have
  \[
    p(F+G)
    = x_1(F+G) \dotsm x_k(F+G)
    = \big( x_1(F) + x_1(G) \big) \dotsm \big( x_k(F) + x_k(G) \big),
  \]
  and the result follows from the fact that $x_i(F) \in \F \langle -\Omega_1 \rangle$ and $x_i(G) \in \F \langle (\Omega_1 - \Omega_1) \setminus -\Omega_1 \rangle$ for all $i \in \{1,\dotsc,k\}$.
\end{proof}

\begin{prop} \label{prop:pgamma-invariant}
  Suppose $p \in S^k(\g)^\g := S^k(\g) \cap S(\g)^\g$ for some $k \in \N_+$.  Then
  \[
    p_\gamma \in S^k(\g \otimes A)^{\g \otimes A}
    \quad \text{for all }
    \gamma \not \in - \Phi_k = \bigcup_{j=1}^k \Big( \big( (\Omega_1 - \Omega_1) \setminus \Omega_1 \big)^j + \Omega_1^{k-j} \Big).
  \]
\end{prop}

\begin{proof}
  Suppose $p \in S^k(\g)^\g$ and define $\Phi_k$ as in \eqref{eq:Upsilonk-def}.  Then, for all $X \in \g \otimes A$ and $F \in \g^* \otimes A^*$, we have
  \begin{align*}
    \sum_{\gamma \in \Gamma} \big( p_\gamma(F) - s (X \cdot p_\gamma) (F) \big) q^{-\gamma}
    &\equiv \sum_{\gamma \in \Omega} p_\gamma(F + s X \cdot F) q^{-\gamma} \mod \F[\Gamma] \otimes s^2 \F[s] \\
    &\equiv p \big( \jmath(F + sX \cdot F) \big) \mod\F[\Gamma] \otimes s^2 \F[s] \\
    &\equiv p \big( \jmath(F) + s \jmath(X \cdot F) \big) \mod \F[\Gamma] \otimes s^2 \F[s] \\
    &\equiv p \big( \jmath(F) + s \imath(X) \cdot \jmath(F) \big) \mod \big( \F \langle \Phi_k \rangle \otimes \F[s] + \F[\Gamma] \otimes s^2 \F[s] \big) \\
    &\equiv p(\jmath(F)) - s( \imath(X) \cdot p )(\jmath(F)) \mod \big( \F \langle \Phi_k \rangle \otimes \F[s] + \F[\Gamma] \otimes s^2 \F[s] \big) \\
    &\equiv p(\jmath(F)) \mod \big( \F \langle \Phi_k \rangle \otimes \F[s] + \F[\Gamma] \otimes s^2 \F[s] \big) \\
    &\equiv \sum_{\gamma \in \Gamma} p_\gamma(F) q^{-\gamma} \mod \big( \F \langle \Phi_k \rangle \otimes \F[s] + \F[\Gamma] \otimes s^2 \F[s] \big)
  \end{align*}
  where the first equivalence uses Lemma~\ref{lem:linearization} and Remark~\ref{rem:zero-p-gamma}, the second and seventh follow from Lemma~\ref{lem:p-jmath-expansion}, the fourth holds by Lemmas~\ref{lem:adjoint-equivariance} and~\ref{lem:polynomial-error-terms}, the fifth follows from Lemma~\ref{lem:linearization}, and the sixth follows from the fact that $p$ is $\g$-invariant.  Therefore, we have $X \cdot p_\gamma =0$, for all $\gamma \not \in -\Phi_k$.
\end{proof}

\begin{rem}
  Note that while the main results of the current paper will assume that $\Omega_1$ has a greatest element,  Proposition~\ref{prop:pgamma-invariant} does \emph{not} require this assumption.  It would be interesting to investigate the properties of the space $S(\g \otimes A)^{\g \otimes A}$ more generally, and to determine, for instance, when the invariant polynomials of Proposition~\ref{prop:pgamma-invariant} generate this space.  Note also that, since $\Omega_1$ is finite, one can always choose an element $\mu \in \Omega$ such that $\omega \le \mu$ for all $\omega \in \Omega_1$.  Then we have the projection
  \[
    \F[t_1,\dotsc,t_\ell]/\F \langle t^\omega \mid \omega \in \Omega,\ \omega \not \le \mu \rangle
    \twoheadrightarrow A,
  \]
  and so representations of $\g \otimes A$ can be pulled back to the setting where we have a greatest element.
\end{rem}

\begin{lem} \label{lem:index-set-max-element}
  If $\Omega_1$ has a greatest element $\mu$, then
  \[
    \Omega_1^k \setminus (-\Phi_k) = k\mu - \Omega_1.
  \]
\end{lem}

\begin{proof}
  If $\Omega_1$ has a greatest element $\mu$, then
  \begin{gather*}
    \Omega_1 = \{\gamma \in \Gamma \mid 0 \le \gamma \le \mu\},\quad
    \Omega_1 - \Omega_1 = \{\gamma \in \Gamma \mid -\mu \le \gamma \le \mu\}, \quad \text{and} \\
    (\Omega_1 - \Omega_1) \setminus \Omega_1 = \{\gamma \in \Gamma \mid -\mu_i \le \gamma_i \le \mu_i \text{ for all } i, \text{ and } \gamma_i<0 \text{ for at least one } i\}.
  \end{gather*}
  Therefore,
  \begin{align*}
    \Omega_1^k &\cap (-\Phi_k)
    = \Omega_1^k \cap \bigcup_{j=1}^k \Big( \big( (\Omega_1 - \Omega_1) \setminus \Omega_1 \big)^j + \Omega_1^{k-j} \Big) \\
    &= \Omega_1^k \cap \bigcup_{j=1}^k \left\{ \gamma \in \Gamma \mid -j\mu_i \le \gamma_i \le k\mu_i - p_i(\mu_i+1) \text{ for all } i, \text{ and some } p_1,\dotsc,p_\ell \text{ with } {\textstyle \sum_i p_i = j} \right\} \\
    &= \left\{ \gamma \in \Gamma \mid 0 \le \gamma_i \le k\mu_i \text{ for all } i, \text{ and } \gamma_i < (k-1) \mu_i \text{ for at least one } i\right\}.
  \end{align*}
  It follows that
  \[
    \Omega_1^k \setminus (-\Phi_k) = \{\omega \in \Omega \mid (k-1)\mu \le \omega \le k \mu\} = k \mu - \Omega_1. \qedhere
  \]
\end{proof}

Suppose $p \in S^k(\g)^\g$, $k \in \N_+$, and assume that $\Omega_1$ has a greatest element $\mu$.  It follows from Proposition~\ref{prop:pgamma-invariant} and Lemma~\ref{lem:index-set-max-element} that $p_\gamma$ lies in $S^k(\g \otimes A)^{\g \otimes A}$ and is (potentially) nonzero for $\gamma$ in the set $k\mu - \Omega_1$.

\subsection{Algebraic independence}

Throughout this subsection, we assume that $\Omega_1$ has a greatest element $\mu$.

\begin{prop} \label{prop:poly-alg-ind}
  A collection of elements
  \begin{equation} \label{eq:algind-g}
    p^{(i)} \in S^{k_i}(\g),\quad 1 \le i \le r,\ k_i \in \N_+,
  \end{equation}
  is algebraically independent if and only if the collection
  \begin{equation} \label{eq:algind-gA}
    p^{(i)}_\omega \in S^{k_i}(\g \otimes A),\quad 1 \le i \le r,\ \omega \in k_i \mu - \Omega_1,
  \end{equation}
  is algebraically independent.
\end{prop}

\begin{proof}
  We prove the result by induction on $\ell$.  If $\ell=1$, then, by Remark~\ref{rem:pgamma-RT}, the result is precisely \cite[Lem.~3.3(ii)]{RT92}.

  Now assume $\ell \ge 2$ and consider a collection of elements \eqref{eq:algind-g}.  Let $\nu = (\mu_1,\dotsc,\mu_{\ell-1}) \in \N^{\ell-1}$.  By Remark~\ref{rem:A-max-element-case}, we have
  \begin{equation} \label{eq:AB-reduction}
    A \cong B \otimes \F[t_\ell]/(t_\ell^{\mu_\ell+1}),\quad \text{where} \quad
    B = \F[t_1,\dotsc,t_{\ell-1}]/ \F \langle t^\omega \mid \omega \not \le \nu \rangle.
  \end{equation}
  By the inductive hypothesis, the collection \eqref{eq:algind-g} is algebraically independent if and only if the collection
  \begin{equation} \label{eq:pi-inductive-step}
    p^{(i)}_{k_i \nu - \lambda} \in S^{k_i}(\g \otimes B),\quad 1 \le i \le r,\ \lambda \in \N^{\ell-1},\ \lambda \le \nu,
  \end{equation}
  is algebraically independent.  Now, note that $\g \otimes B$ is a finite-dimensional Lie algebra.  Therefore, the $\ell=1$ case of the proposition implies that the collection \eqref{eq:pi-inductive-step} is algebraically independent if and only if
  \[
    \left( p^{(i)}_{k_i \nu - (\lambda_1,\dotsc,\lambda_{\ell-1})} \right)_{k_i \mu_\ell - \lambda_\ell} \in S^{k_i}(\g \otimes A),\quad 1 \le i \le r,\ \lambda \in \Omega_1,
  \]
  are algebraically independent.  By Lemma~\ref{lem:double-invariant-poly}, we have
  \[
    \left( p^{(i)}_{k_i \nu - (\lambda_1,\dotsc,\lambda_{\ell-1})} \right)_{k_i \mu_\ell - \lambda_\ell} = p^{(i)}_{k_i \mu - \lambda} \quad \text{ for all $i$ and $\lambda$},
  \]
  completing the proof of the inductive step.
\end{proof}

\begin{rem} \label{rem:generator-independent}
  It is clear from the definitions in Section~\ref{subsec:poly-defs} that, if $p \in S^k(\g \otimes A)$ is a product of the polynomials \eqref{eq:algind-g}, then, for $\omega \in \Omega_1$, $p_{k \mu - \omega}$ is a polynomial in the elements \eqref{eq:algind-gA}.  Thus, the subalgebra $T_A$ of $S(\g \otimes A)$ generated by the polynomials \eqref{eq:algind-gA} depends only on the subalgebra $T$ of $S(\g)$ generated by the polynomials \eqref{eq:algind-g}, and not on the individual polynomials themselves.
\end{rem}

%
\section{Regular elements and bilinear forms}
%

\subsection{Regular elements}

Suppose $\fa$ is a finite-dimensional Lie algebra.  For $f \in \fa^*$, we define
\[
  \fa^f = \{x \in \fa \mid x \cdot f = 0\} = \{x \in \fa \mid f([\fa,x])=0 \}.
\]
The \emph{index} of $\fa$ is defined to be
\[
  \chi(\fa) = \Inf \{\dim \fa^f \mid f \in \fa^*\}.
\]
We say that $f$ is \emph{regular} if $\dim \fa^f = \chi(\fa)$.

For an element $x \in \fa$, we let
\[
  \fa^x = \{y \in \fa \mid [x,y]=0\}
\]
be the centralizer of $x$ in $\fa$.  An element of $\fa$ is said to be \emph{regular} if its centralizer has minimal dimension.  (If $\fa^*$ is identified with $\fa$ via an invariant nondegenerate bilinear form, this corresponds to the notion of regular element of $\fa^*$ given above.)

\begin{prop} \label{prop:chi-regular}
  Let $\g$ be a finite-dimensional Lie algebra, $\mu \in \N^\ell$, and define
  \[
    A = \F[t_1,\dotsc,t_\ell] / \F \langle t^\omega \mid \omega \not \le \mu \rangle.
  \]
  \begin{enumerate}
    \item \label{prop-item:chi} We have $\chi(\g \otimes A) = \chi(\g) \dim A$.

    \item \label{prop-item:regular} A linear form $F = \sum_{\omega \le \mu} f_\omega \otimes \epsilon^\omega$ on $\g \otimes A$ is regular if and only if $f_\mu$ is a regular linear form on $\g$.
  \end{enumerate}
\end{prop}

\begin{proof}
  We prove the result by induction on $\ell$.  If $\ell=1$, the theorem is precisely \cite[Th.~2.8]{RT92}.  Now assume $\ell \ge 2$, and let $\nu = (\mu_1,\dotsc,\mu_{\ell-1}) \in \N^{\ell-1}$.  As in the proof of Proposition~\ref{prop:poly-alg-ind}, we have the isomorphism \eqref{eq:AB-reduction}.  By the inductive hypothesis, we have
  \[
    \chi(\g \otimes A) =\chi(\g \otimes B)(\mu_\ell + 1) = \chi(\g) (\dim B) (\mu_\ell + 1) = \chi(\g) \dim A,
  \]
  completing the inductive step in the proof of part~\eqref{prop-item:chi}.

  We now prove part~\eqref{prop-item:regular}.  Using the isomorphism \eqref{eq:AB-reduction} to identify $A$ with $B \otimes \F[t_\ell]/(t_\ell^{\mu_\ell +1})$, we write
  \[
    F = \sum_{\omega \le \mu} \left( f_\omega \otimes \epsilon^{(\omega_1,\dotsc,\omega_{\ell-1})} \right) \otimes \epsilon_\ell^{\omega_\ell},
  \]
  where $\epsilon_\ell^n$, $0 \le n \le \mu_\ell$ is the dual basis to the basis $\tau_\ell^n$, $0 \le n \le \mu_\ell$ of $\F[t_\ell]/(t_\ell^{\mu_\ell+1})$.  Then, by the $\ell=1$ case of the theorem, the linear form $F$ on $\g \otimes A$ is regular if and only if
  \[
    \sum_{\omega \le \nu} f_{(\omega_1,\dotsc,\omega_{\ell-1},\mu_\ell)} \otimes \epsilon^{(\omega_1,\dotsc,\omega_{\ell-1})}
  \]
  is a regular form on $\g \otimes B$.  By the inductive hypothesis, this form is regular if and only if $f_\mu$ is a regular form on $\g$.  This completes the proof of the inductive step.
\end{proof}

\subsection{Bilinear forms}

In this subsection, we assume that $\Omega_1$ has a greatest element $\mu$.  Suppose
\[
  \kappa \colon \g \times \g \to \F
\]
is a bilinear form on $\g$.  We extend $\kappa$ to a bilinear map
\begin{equation} \label{eq:kappa-Omega-def}
  \kappa^\Omega \colon \left( \g \otimes \F[\Omega] \right) \times \left( \g \otimes \F[\Omega] \right) \to \F[\Omega],\quad
  \kappa^\Omega (x \otimes q^\omega, y \otimes q^\gamma) = \kappa(x,y) q^{\omega + \gamma},
\end{equation}
extended by bilinearity.  For any $\omega \in \Omega$, we can compose $\kappa^\Omega$ with the projection onto the $q^\omega$-component to obtain a bilinear form $\kappa^\Omega_\omega$.  More precisely, we have
\begin{equation} \label{eq:kappa-Omega-decomp}
  \kappa^\Omega = \sum_{\omega \in \Omega} \kappa_\omega^\Omega q^\omega,
\end{equation}
where, for $\omega \in \Omega$, we have
\[
  \kappa_\omega^\Omega \colon (\g \otimes \F[\Omega]) \times (\g \otimes \F[\Omega]) \to \F,\quad
  \kappa_\omega^\Omega \left( \sum_{\nu \in \Omega} x_\nu \otimes q^\nu, \sum_{\nu \in \Omega} y_\nu \otimes q^\nu \right) = \sum_{\nu + \nu' = \omega} \kappa(x_\nu,y_{\nu'}).
\]
Using the injection $\imath$ of \eqref{eq:imath-def}, we have the corresponding bilinear forms on $\g \otimes A$.  Precisely, for $\omega \in \Omega$, we have
\[
  \kappa_\omega \colon (\g \otimes A) \times (\g \otimes A) \to \F,\quad
  \kappa_\omega \left( \sum_{\nu \in \Omega_1} x_\nu \otimes \tau^\nu, \sum_{\nu \in \Omega_1} y_\nu \otimes \tau^\nu \right) = \sum_{\nu + \nu' = \omega} \kappa(x_\nu,y_{\nu'}).
\]

\begin{lem} \label{lem:bilinear-form-properties}
  The bilinear forms defined above have the following properties.
  \begin{enumerate}
    \item \label{lem-item:form-symmetric} The forms $\kappa^\Omega_\omega$ are symmetric for all $\omega \in \Omega$ if and only if $\kappa$ is symmetric.
    \item \label{lem-item:form-invariant} The forms $\kappa^\Omega_\omega$ are invariant for all $\omega \in \Omega$ if and only if $\kappa$ is invariant.
    \item \label{lem-item:form-nondegen} The form $\kappa$ is nondegenerate if and only if $\kappa_\mu$ is nondegenerate.
  \end{enumerate}
\end{lem}

\begin{proof}
  Part~\eqref{lem-item:form-symmetric} follows immediately from \eqref{eq:kappa-Omega-decomp} and \eqref{eq:kappa-Omega-def}.
  \details{
    Suppose $\kappa$ is symmetric.  Then $\kappa^\Omega$ is symmetric by \eqref{eq:kappa-Omega-def}.  Hence the $\kappa^\Omega_\omega$ are symmetric by \eqref{eq:kappa-Omega-decomp}.  Now suppose the $\kappa^\Omega_\omega$ are symmetric.  Then $\kappa^\Omega$ is symmetric by \eqref{eq:kappa-Omega-decomp}.  Since $\kappa$ is simply the restriction of $\kappa^\Omega$ to $\g \subseteq \g \otimes \F[\Omega]$, it follows that $\kappa$ is symmetric.
  }

  To prove part~\eqref{lem-item:form-invariant}, suppose $\kappa$ is invariant.  Then, for all $x,y,z \in \g$ and $\omega, \nu, \lambda \in \Omega$, we have
  \[
    \kappa^\Omega([x \otimes q^\omega,y \otimes q^\nu],z \otimes q^\lambda) - \kappa^\Omega(x \otimes q^\omega, [y \otimes q^\nu, z \otimes q^\lambda])
    = \big( \kappa([x,y],z) - \kappa(x,[y,z]) \big) q^{\omega+\nu+\lambda}
    = 0.
  \]
  Thus, $\kappa^\Omega$ is invariant.  Conversely, suppose $\kappa^\Omega$ is invariant.  Then, since $\kappa$ is simply the restriction of $\kappa^\Omega$ to $\g$, it follows that $\kappa$ is invariant.  The fact that the forms $\kappa^\Omega_\omega$, $\omega \in \Omega$, are invariant if and only if $\kappa^\Omega$ is invariant follows immediately from \eqref{eq:kappa-Omega-decomp}.  This completes the proof of part~\eqref{lem-item:form-invariant}.

  To prove part~\eqref{lem-item:form-nondegen}, assume $\kappa$ is nondegenerate and fix an arbitrary nonzero element $\sum_{\omega \in \Omega_1} x_\omega \otimes \tau^\omega \in \g \otimes A$.  Choose $\nu$ such that $x_\nu \ne 0$.  Since $\kappa$ is nondegenerate, there exists $y \in \g$ such that $\kappa(x_\nu,y) \ne 0$.  Since $\mu$ is a greatest element, we have $\mu - \nu \in \Omega_1$, and
  \[
    \kappa_\mu \left( \sum_{\omega \in \Omega_1} x_\omega \otimes \tau^\omega, y \otimes \tau^{\mu-\nu} \right)
    = \kappa(x_\nu,y) \ne 0.
  \]
  Thus $\kappa_\mu$ is nondegenerate.  Conversely, assume $\kappa_\mu$ is nondegenerate and fix an arbitrary nonzero element $x \in \g$.  Then there exists $y \otimes \tau^\mu \in \g \otimes A$ such that
  \[
    \kappa(x,y) = \kappa_\mu(x,y \otimes \tau^\mu) \ne 0.
  \]
  So $\kappa$ is nondegenerate.
\end{proof}

Recall that a quadratic Lie algebra is a Lie algebra $\g$, together with a symmetric, nondegenerate, invariant, bilinear form.

\begin{cor}
  If $\g$ is a quadratic Lie algebra, then so is $\g \otimes A$.
\end{cor}

\begin{proof}
  It follows immediately from Lemma~\ref{lem:bilinear-form-properties} that if $\kappa$ is a symmetric, nondegenerate, invariant, bilinear form on $\g$, then $\kappa_\mu$ is a symmetric, nondegenerate, invariant, bilinear form on $\g \otimes A$.
\end{proof}

\begin{rem} \label{rem:pairing-identification}
  Suppose $\kappa$ is a symmetric, nondegenerate, invariant, bilinear form on $\g$.  Via the form $\kappa_\mu$, we can identify elements of $\g \otimes A_\omega$ with elements of $(\g \otimes A)^*_{\mu - \omega}$, for all $\omega \in \Omega_1$.  In particular, we have the following.
  \begin{enumerate}
    \item \label{rem-item:regular} By Proposition~\ref{prop:chi-regular}\eqref{prop-item:regular}, an element $\sum_{\omega \in \Omega_1} x_\omega \otimes \tau^\omega$ of $\g \otimes A$ is regular if and only if $x_0$ is a regular element of $\g$.
        \details{
          Let $X = \sum_{\omega \in \Omega_1} x_\omega \otimes \tau^\omega$ and $F = \sum_{\omega \in \Omega_1} f_\omega \otimes \epsilon^\omega \in (\mathfrak g \otimes A)^*$ denote the linear form $\kappa_\mu (X , -)$.  Since $\kappa$ is bilinear, we have $f_\omega = \kappa (x_{\mu-\omega}, -)$ for all $\omega \in \Omega_1$.  Moreover,  since $\kappa_\mu$ is nondegenerate and invariant, we have $(\mathfrak g \otimes A)^F = (\mathfrak g \otimes A)^X$.  By Proposition~\ref{prop:chi-regular}\eqref{prop-item:regular}, $F$ is regular; that is, $(\mathfrak g \otimes A)^F$ has dimension $\chi(\mathfrak g \otimes A)$, if and only if $f_\mu = \kappa(x_0, -)$ is regular.  Since $\kappa$ is also nondegenerate and invariant, we have $\mathfrak g^{f_\mu} = \mathfrak g^{x_0}$.  Thus $f_\mu$ is regular if and only if $x_0$ is regular.
        }

    \item \label{rem-item:polys} Suppose $p^{(j)} \in S^{k_j}(\g^*)^\g$ for $j \in \{1,\dotsc,r\}$.  Then, by Proposition~\ref{prop:pgamma-invariant} and Lemma~\ref{lem:index-set-max-element}, the polynomials $p^{(j)}_\omega \in S^{k_j} \big( (\g \otimes A)^* \big)$, $\omega \in \Omega_1$, $j \in \{1,\dotsc,r\}$, lie in $S \big( (\g \otimes A)^* \big)^{\g \otimes A}$.  By Proposition~\ref{prop:poly-alg-ind}, the $p^{(j)}_\omega$ are algebraically independent if and only if the $p^{(j)}$ are algebraically independent.
  \end{enumerate}
\end{rem}

%
\section{The semisimple case}
%

In this section we assume $\g$ is semisimple and $\F$ is algebraically closed (still of characteristic zero). Recall that a \emph{regular} element of $\g$ is an element whose centralizer has the minimal dimension $\chi(\g)$.  An $\fsl_2$-triple is said to be \emph{principal} if its elements are regular.  Fix a principal $\fsl_2$-triple $(x_+,x_-,h)$ of $\g$.  (For a proof that principal $\fsl_2$-triples always exists, see, for instance, \cite[Ch.~VIII, \S11, no.~4, Prop.~8]{Bou05}.)  We let
\[
  \ft_0 = x_+ + \g^{x_-}
\]
be the associated transversal slice.  We collect some important well-known properties of this slice in the following lemma.

\begin{lem} \label{lem:slice-properties}
  The transversal slice defined above has the following properties.
  \begin{enumerate}
    \item \label{lem-item:slice-regular} Every element of $\ft_0$ is a regular element of $\g$.

    \item \label{lem-item:orbits-meet-slice} The orbit under the adjoint action of every regular element of $\g$ intersects $\ft_0$ at a unique point, and this intersection is transversal.  That is, we have
      \[
        \g = \g^{x_-} \oplus [\g,x] \quad \text{for all } x \in \ft_0.
      \]

    \item \label{lem-item:restriction-isom} Let $R_0 \colon S(\g^*) \to S(\ft_0^*) = \F[\ft_0]$ be the operation of restriction of polynomial functions on $\g$ to polynomial functions on $\ft_0$.  Then the restriction of $R_0$ to invariant polynomials yields an isomorphism of associative algebras $S(\g^*)^\g \cong \F[\ft_0]$.
  \end{enumerate}
\end{lem}

\begin{proof}
  A proof of part~\eqref{lem-item:slice-regular} can be found in \cite[Lem.~10 and Lem.~12]{Kos63}.  Part~\eqref{lem-item:orbits-meet-slice} follows from \cite[Th.~8]{Kos63}.  Part~\eqref{lem-item:restriction-isom} is \cite[Th.~7]{Kos63}.  Although the paper \cite{Kos63} works over the field of complex numbers, the proofs are valid over arbitrary algebraically closed fields of characteristic zero.
\end{proof}

It is clear that $\g^{x_-} \otimes A$ is the centralizer of $x_- \otimes 1$ in $\g \otimes A$.  We define the affine space
\[
  \ft = x_+ \otimes 1 + \g^{x_-} \otimes A \subseteq \g \otimes A.
\]
By Lemma~\ref{lem:slice-properties}\eqref{lem-item:slice-regular} and Remark~\ref{rem:pairing-identification}\eqref{rem-item:regular}, $\ft$ consists of regular elements of $\g \otimes A$.  We let
\[
  R \colon S \big( (\g \otimes A)^* \big) \to S(\ft^*) = \F[\ft]
\]
denote the restriction of polynomial functions on $\g \otimes A$ to polynomial functions on $\ft$.  The following result is a generalization of \cite[Lem.~4.2]{RT92}.

\begin{prop} \label{prop:R-injection}
  \begin{enumerate}
    \item \label{prop-item:adjoint-orbits-meet-t} Let $X$ be a regular element of $\g \otimes A$.  Then the orbit of $X$ under the adjoint action meets $\ft$.

    \item \label{prop-item:R-injection} The restriction of $R$ to $S \big( (\g \otimes A)^* \big)^{\g \otimes A}$ yields an injection of $S \big( (\g \otimes A)^* \big)^{\g \otimes A}$ into $\F[\ft]$.
  \end{enumerate}
\end{prop}

\begin{proof}
  Part~\eqref{prop-item:R-injection} clearly follows from part~\eqref{prop-item:adjoint-orbits-meet-t}.  To prove part~\eqref{prop-item:adjoint-orbits-meet-t}, let $X = \sum_{\omega \in \Omega_1} x_\omega \otimes \tau^\omega$ be a regular element of $\g \otimes A$.  By Remark~\ref{rem:pairing-identification}\eqref{rem-item:regular}, $x_0$ is a regular element of $\g$.  By Lemma~\ref{lem:slice-properties}\eqref{lem-item:orbits-meet-slice}, there exists an element $g$ of the adjoint group (i.e.\ group of inner automorphisms) of $\g$ such that $g(x_0) \in \ft_0$.  Since the linear map $g \otimes \id_A$ is an element of the adjoint group of $\g \otimes A$, we may assume without loss of generality that $x_0 \in \ft_0$.  Therefore, by Lemma~\ref{lem:slice-properties}\eqref{lem-item:orbits-meet-slice}, we have
  \begin{equation} \label{eq:R-injection-g-decomp}
    \g = \g^{x_-} \oplus [\g,x_0].
  \end{equation}

  To complete the proof, we want to show that, perhaps after acting by some element in the adjoint group, we have $x_\omega \in \g^{x_-}$ for all $\omega \in \Omega_1 \setminus \{0\}$.  Let $\Omega_1'$ be a subset of $\Omega_1 \setminus \{0\}$ such that
  \begin{gather}
    x_\omega \in \g^{x_-} \quad \text{for all } \omega \in \Omega_1',\quad \text{and} \label{orbit-cond1} \\
    \omega \in \Omega_1',\ 0 < \lambda \le \omega \implies \lambda \in \Omega_1'. \label{orbit-cond2}
  \end{gather}
  If $\Omega_1' = \Omega_1 \setminus \{0\}$, we are done.  Therefore, assume $\Omega_1' \ne \Omega_1 \setminus \{0\}$, and let $\nu$ be a minimal element of $\Omega_1 \setminus \big( \Omega_1' \cup \{0\} \big)$.  By \eqref{eq:R-injection-g-decomp}, there exists $z \in \g$ such that $x_\nu + [z,x_0] \in \g^{x_-}$.  Then we have
  \begin{align*}
    X' := \exp \big( \ad (z \otimes \tau^\nu) \big) (X)
    &= \sum_{\omega \not \ge \nu} x_\omega \otimes \tau^\omega + \big( x_\nu + [z,x_0] \big) \otimes \tau^\nu \mod \g \otimes A_{> \nu} \\
    &= \sum_{\omega \not \ge \nu} x_\omega \otimes \tau^\omega + x_\nu' \otimes \tau^\nu \mod \g \otimes A_{> \nu},
  \end{align*}
  for some $x_\nu' \in \g^{x_-}$.  If $X' = \sum_{\omega \in \Omega_1} x_\omega' \otimes \tau^\omega$, then $x'_\omega = x_\omega$ for all $\omega \not \ge \nu$ (in particular, for all $\omega \in \Omega_1'$), and so $\Omega_1' \cup \{\nu\}$ satisfies \eqref{orbit-cond1} and \eqref{orbit-cond2}.  Since $\Omega_1$ is finite, part~\eqref{prop-item:adjoint-orbits-meet-t} of the proposition follows by induction on the size of $\Omega_1'$ (the base case being $\Omega_1' = \varnothing$).
\end{proof}

For the remainder of this section, we assume that
\begin{center}
  $\Omega_1$ has a greatest element $\mu$.
\end{center}

By \cite[Th.~6]{Kos63} we can choose a basis $(u_1,\dotsc,u_r)$ of $\g^{x_-}$ and a system $(p^{(1)},\dotsc,p^{(r)})$ of homogeneous, algebraically independent generators of $S(\g^*)^\g$ such that $p^{(j)} \in S^{k_j}(\g^*)^\g$, $k_j \in \N_+$, and
\begin{equation}
  p^{(j)} \left( x_+ + \sum_{i=1}^r c_i u_i \right) = c_j,\quad \text{for all } c_1,\dotsc,c_r \in \F.
\end{equation}
\details{
  For the notation used in \cite[Th.~6]{Kos63}, see \cite[Th.~5]{Kos63} and the beginning of \cite[\S4.6]{Kos63}.
}
In other words
\[
  p^{(j)}(x) = u_j^*(x - x_+) \quad \text{for all } x \in \ft_0,
\]
where $(u_j^*)_j$ is the dual basis to $(u_j)_j$.

Consider the basis of $\g^{x_-} \otimes A$
\[
  (U_{i,\omega})_{i \in \{1,\dotsc,r\},\, \omega \in \Omega_1},\quad U_{i,\omega} = u_i \otimes \tau^\omega,
\]
and let $(U_{i,\omega}^*)$ denote the corresponding dual basis.  Then we have a system of coordinates $(\varepsilon_{i,\omega})_{i \in \{1,\dotsc,r\},\, \omega \in \Omega_1}$ on $\ft$ given by
\begin{equation} \label{eq:varepsilon-U*}
  \varepsilon_{i,\omega}(X) = U_{i,\omega}^* (X - x_+ \otimes 1),\quad i \in \{1,\dotsc,r\},\ \omega \in \Omega_1.
\end{equation}
Thus, we have
\[
  X = x_+ \otimes 1 + \sum_{\omega \in \Omega_1} \sum_{i=1}^r \varepsilon_{i,\omega}(X) u_i \otimes \tau^\omega,\quad \text{for all } X \in \ft.
\]

Suppose $V$ is a vector space and $f \in V^*$.  Then we can endow $V$ with the structure of an $S(V)$-module by defining
\[
  p \cdot v = p(f)v,\quad p \in S(V),\ v \in V.
\]
Let $D_f$ be the unique $f$-derivation of $S(V)$ into $V$ satisfying $D_f(v) = v$ for all $v \in V$.  (Thus, the Leibniz rule is $D_f(xy) = D_f(x) f(y) + f(x) D_f(y)$.)  More explicitly, $D_f$ is the unique linear map $S(V) \to V$ such that $D_f(\F) = 0$ and
\[
  D_f \big( v_1^{n_1} \dotsm v_k^{n_k} \big)
  = \sum_{i=1}^k n_i f(v_1)^{n_1} \dotsm f(v_{i-1})^{n_{i-1}} f(v_i)^{n_i-1} f(v_{i+1})^{n_{i+1}} \dotsm f(v_k)^{n_k} v_i,
\]
for all $v_1,\dotsc,v_k \in V$ and $n_1,\dotsc,n_k \in \N_+$.  Note that for $p \in S(V)$ and $f, g \in V^*$, we have
\begin{equation} \label{eq:derivative-expression}
  \big( D_f(p) \big)(g) = \left. \frac{d}{dt} \right|_{t=0} p(f + tg).
\end{equation}

The following result is a generalization of \cite[Lem.~4.4]{RT92}.

\begin{lem} \label{lem:restriction}
  For all $i,j \in \{1,\dotsc,r\}$, $\lambda, \omega \in \Omega_1$, and $X \in \ft$, we have
  \begin{gather}
    R \left( p^{(j)}_\omega \right) = \varepsilon_{j,\omega},\quad \text{and} \label{eq:restriction-coordinates} \\
    \left( D_X \left( p^{(j)}_\omega \right) \right) (U_{i,\lambda}) = \delta_{i,j} \delta_{\lambda,\omega}. \label{eq:derivative-coordinates}
  \end{gather}
\end{lem}

\begin{proof}
  We have
  \[
    p^{(j)} \big( \imath (x_+ \otimes 1 + u_i \otimes \tau^\lambda) \big)
    = p^{(j)} \big( x_+ \otimes 1 + u_i \otimes q^\lambda \big)
    = \delta_{i,j} q^\lambda.
  \]
  \details{
    In the expression $p^{(j)} \big( x_+ \otimes 1 + u_i \otimes q^\lambda \big)$, we view $p^{(j)}$ as a polynomial function $\g \otimes \F[\Omega] \to \F[\Omega]$, as explained in Section~\ref{subsec:poly-funcs-adjoint}.  Viewing $\F[\Omega]$ as a polynomial algebra in $q_1,\dotsc,q_\ell$ as described before Corollary~\ref{cor:p-gamma-deriv-expression}, we can evaluate elements of $\F[\Omega]$ at $a=(a_1,\dotsc,a_\ell) \in \F^\ell$.  Let $\ev_a$ denote this evaluation.  Then, we have
    \[
      p^{(j)} \circ \ev_a = \ev_a \circ p^{(j)}.
    \]
    Therefore,
    \[
      \ev_a \circ p^{(j)} \big( x_+ \otimes 1 + u_i \otimes q^\lambda \big)
      = p^{(j)} \circ \ev_a \big( x_+ \otimes 1 + u_i \otimes q^\lambda \big)
      = p^{(j)} \big( x_+ + a^\lambda u_i \big)
      = \delta_{i,j} a^\lambda
      = \ev_a \big( \delta_{i,j} q^\lambda \big).
    \]
    It follows that $p^{(j)} \big( x_+ \otimes 1 + u_i \otimes q^\lambda \big) = \delta_{i,j} q^\lambda$ since two functions $p,p'$ on $\ft_0 \otimes \F[\Omega]$ are equal if and only if $\ev_a \circ p$ and $\ev_a \circ p'$ are equal as functions on $\ft_0$.
  }
  It then follows from Lemma~\ref{lem:p-jmath-expansion} that
  \[
    p_\omega^{(j)} \big( x_+ \otimes 1 + u_i \otimes \tau^\lambda \big)
    = \delta_{\omega,\lambda} \delta_{i,j},
  \]
  proving \eqref{eq:restriction-coordinates}.

  To prove \eqref{eq:derivative-coordinates}, we compute
  \begin{multline*}
    \left( D_X \left( p^{(j)}_\omega \right) \right) (U_{i,\lambda})
    \stackrel{\eqref{eq:derivative-expression}}{=} \left. \frac{d}{dt} \right|_{t=0} p^{(j)}_\omega (X + t U_{i,\lambda} )
    \stackrel{\eqref{eq:restriction-coordinates}}{=} \left. \frac{d}{dt} \right|_{t=0} \varepsilon_{j,\omega} (X + t U_{i,\lambda} ) \\
    \stackrel{\eqref{eq:varepsilon-U*}}{=} \left. \frac{d}{dt} \right|_{t=0} U_{j,\omega}^* (X  - x_+ \otimes 1 + t U_{i,\lambda} )
    = \delta_{i,j} \delta_{\lambda,\omega}. \qedhere
  \end{multline*}
\end{proof}

We can now state our main result.  The special case $\ell=1$ was established in \cite[Th.~4.5]{RT92}.

\begin{theo} \label{theo:inv-polys}
  Recall that $\g$ is a semisimple Lie alegbra, $\F$ is an algebraically closed field of characteristic zero, $\mu \in \N^\ell$, and $A = \F[t_1,\dotsc,t_\ell] / (t_1^{\mu_1+1},\dotsc,t_\ell^{\mu_\ell+1})$.
  \begin{enumerate}
    \item \label{theo-item:restriction-isom} The restriction $R$ induces an isomorphism of algebras $S \big( (\g \otimes A)^* \big)^{\g \otimes A} \to S(\ft^*) = \F[\ft]$.

    \item \label{theo-item:inv-poly-gen-set} The polynomials $p_\omega^{(j)} \in S^{k_j}\big( (\g \otimes A)^* \big)$, $\omega \in \Omega_1$, $j \in \{1,\dotsc,r\}$, form a system of algebraically independent generators of $S \big( (\g \otimes A)^* \big)^{\g \otimes A}$.

    \item The orbit under the adjoint action of every regular element of $\g \otimes A$ intersects $\ft$ at a unique point, and this intersection is transversal.  That is, we have
      \[
        \g \otimes A = \left( \g^{x_-} \otimes A \right) \oplus [\g \otimes A, X]
        \quad \text{for all } X \in \ft.
      \]
  \end{enumerate}
\end{theo}

\begin{proof}
  \begin{asparaenum}
    \item Since $p_\omega^{(j)} \in S \big( (\g \otimes A)^* \big)^{\g \otimes A}$ for all $\omega \in \Omega_1$ and $j \in \{1,\dotsc,r\}$, it follows from Lemma~\ref{lem:restriction} that $R \left( S \big( (\g \otimes A)^* \big)^{\g \otimes A} \right) = \F[\ft]$.  Proposition~\ref{prop:R-injection}\eqref{prop-item:R-injection} then implies that $R$ induces an isomorphism of algebras $S \big( (\g \otimes A)^* \big)^{\g \otimes A} \to S(\ft^*) = \F[\ft]$.

    \item This follows immediately from part~\eqref{theo-item:restriction-isom} and Remark~\ref{rem:pairing-identification}\eqref{rem-item:polys}.

    \item Suppose $X$ is a regular element of $\g \otimes A$.  By Proposition~\ref{prop:R-injection}\eqref{prop-item:adjoint-orbits-meet-t}, there exists a point $X'$ in the intersection of $\ft$ and the orbit of $X$ under the adjoint action.  Then we have
      \[
        p_\omega^{(j)}(X) = p_\omega^{(j)}(X') \stackrel{\eqref{eq:restriction-coordinates}}{=} \varepsilon_{j,\omega}(X')
        \quad \text{for all } \omega \in \Omega_1,\ j \in \{1,\dotsc,r\},
      \]
      where the first equality holds since $p_\omega^{(j)}$ is $\g \otimes A$-invariant by Remark~\ref{rem:pairing-identification}\eqref{rem-item:polys}.  It follows that
      \[
        X' = x_+ \otimes 1 + \sum_{\omega \in \Omega_1} \sum_{j=1}^r p_\omega^{(j)}(X) u_j \otimes \tau^\omega,
      \]
      and so the intersection point $X'$ is unique.

      It remains to prove transversality.  Choose $X \in \ft$.  Then $X$ is a regular element of $\g \otimes A$, and so we have
      \[
        \dim [\g \otimes A, X] + \dim \big( \g^{x_-} \otimes A \big) = \dim \g \otimes A.
      \]
      \details{
        Consider the adjoint action of $X$ on $\g \otimes A$.  The image of this action is $[\g \otimes A, X]$, and the kernel is $(\g \otimes A)^X$.  Thus, by the Rank-Nullity Theorem, we have
        \[
          \dim [\g \otimes A, X] + \dim (\g \otimes A)^X = \dim \g \otimes A.
        \]
        By Remark~\ref{rem:pairing-identification}\eqref{rem-item:regular}, $x_- \otimes 1$ is a regular element of $\g \otimes A$.  Since $X$ is also a regular element of $\g \otimes A$, we have $\dim (\g \otimes A)^X = \dim (\g \otimes A)^{x_- \otimes 1} = \dim \g^{x_-} \otimes A$.  The claim follows.
      }
      Therefore, it suffices to prove that $[\g \otimes A,X] \cap \big( \g^{x_-} \otimes A \big) = \{0\}$.  Suppose $Y \in [\g \otimes A,X] \cap \big( \g^{x_-} \otimes A \big)$.  Then there exist $c_{j,\omega} \in \F$ such that $Y = \sum_{j,\omega} c_{j, \omega} U_{j,\omega}$.  On the other hand, every $P \in S \big( (\g \otimes A)^* \big)^{\g \otimes A}$ is constant on orbits under the adjoint action, and thus $(D_X P)(Y) = 0$, since $Y \in [\g \otimes A,X]$.
      \details{
        For all $t \in \F$, since $tY \in [\g \otimes A,X]$, there is an element $g_t$ of the adjoint group of $\g \otimes A$ such that $g_t \cdot X = X + tY + O(t^2)$.  Then, by \eqref{eq:derivative-expression}, we have
        \[
          (D_X P)(Y)
          = \left. \frac{d}{dt} \right|_{t=0} P(X + tY)
          = \left. \frac{d}{dt} \right|_{t=0} P(g_t \cdot X)
          = \left. \frac{d}{dt} \right|_{t=0} P(X)
          = 0.
        \]
      }
      Therefore, for all $\lambda \in \Omega_1$, $p \in S(\g^*)^\g$, and $i \in \{1,\dotsc,r\}$, we have
      \[
        0
        = \left( D_X \left( p_\lambda^{(i)} \right) \right) (Y)
        = \sum_{j,\omega} c_{j,\omega} \left( D_X \left( p_\lambda^{(i)} \right) \right) \left( U_{j,\omega} \right)
        = c_{i,\lambda},
      \]
      where the last equality follows from \eqref{eq:derivative-coordinates}.  Thus, $Y=0$, completing the proof. \qedhere
  \end{asparaenum}
\end{proof}

The Killing form $\kappa$ induces an isomorphism of vector spaces $K_\kappa \colon \g \to \g^*$ given by $K_\kappa (x) = \kappa (x, -)$.  This induces an algebra isomorphism $K_\kappa \colon S(\g) \to S(\g^*)$.  For each $j \in \{ 1, \dotsc, r \}$, let $\rho^{(j)} = K_\kappa^{-1} \left( p^{(j)} \right) \in S(\g)^\g$. Theorem~\ref{theo:inv-polys}\eqref{theo-item:inv-poly-gen-set} then implies the following.

\begin{cor} \label{cor:inv-poly-gen-set}
  With notation as in Theorem~\ref{theo:inv-polys}, the polynomials $\rho_\omega^{(j)} \in S^{k_j} (\g \otimes A)$, $\omega \in k_j \mu - \Omega_1$, $j \in \{ 1, \dotsc, r \}$, form a system of algebraically independent generators of $S (\g \otimes A)^{\g \otimes A}$.
\end{cor}

\begin{eg}
  Consider the case $\g = \fsl_2$, and let $\{x_+,x_-,h\}$ be the standard Chevalley basis.  The Casimir element $c = \frac{1}{64} h^2 + \frac{1}{32} x_- x_+ + \frac{1}{32} x_+ x_-$, which generates the center of $U(\g)$ as an algebra, corresponds to the polynomial function $\rho = \frac{1}{64} \left( h^2 + 4 x_- x_+ \right) \in S(\g)^\g$.  Now, given $\ell > 0$ and $\mu \in \N^\ell$, let $A = \F [t_1, \dotsc, t_\ell] / \F \langle t^\omega \mid \omega \not\le \mu \rangle$.  By Corollary~\ref{cor:inv-poly-gen-set}, the polynomials
  \[
    \rho_{2\mu - \omega}
    = \sum_{0 \le \gamma \le \omega} \big( (h \otimes \tau^{\mu - \gamma}) (h \otimes \tau^{\mu - \omega + \gamma}) + 4 (x_- \otimes \tau^{\mu - \gamma}) (x_+ \otimes \tau^{\mu - \omega + \gamma}) \big),\quad \omega \in \N^\ell,\ \omega \le \mu,
  \]
  form a system of algebraically independent generators of $S (\g \otimes A)^{\g \otimes A}$.

  Now, the dual basis $\{ x_+^*, x_-^*, h^*\}$  of $\g^*$ is identified with $\left\{ \frac 1 4 x_-, \frac 1 4 x_+, \frac 1 8 h \right\}$ via the Killing form, and the Casimir element $c$ corresponds to the polynomial function $p = (h^*)^2 + x_-^* x_+^* \in S(\g^*)^\g$ (see \cite[Example~23.3]{Hum78}).  By Theorem~\ref{theo:inv-polys}\eqref{theo-item:inv-poly-gen-set}, the polynomials
  \[
    p_{\omega}
    = \sum_{0 \le \gamma \le \omega} \big( (h^* \otimes \epsilon^\gamma) (h^* \otimes \epsilon^{\omega - \gamma}) + (x_-^* \otimes \epsilon^\gamma) (x_+^* \otimes \epsilon^{\omega - \gamma}) \big),\quad \omega \in \N^\ell,\ \omega \le \mu,
  \]
  form a system of algebraically independent generators of $S\big( (\g \otimes A)^* \big)^{\g \otimes A}$.
\end{eg}

As an application of Theorem~\ref{theo:inv-polys}, we have the following result.

\begin{theo} \label{theo:center-acts-trivially}
  If $\mu > 0$ (i.e.\ $A \ne \F$), then the center of $U(\g \otimes A)$ acts by the restriction of the augmentation map $U(\g \otimes A) \to \F$ on any irreducible finite-dimensional module.  In particular, all irreducible finite-dimensional modules have the same central character.
\end{theo}

\begin{proof}
  Let $\varphi$ be an irreducible finite-dimensional representation of $\g \otimes A$.  By \cite[Prop.~10]{CFK10} or \cite[Cor.~5.8]{NSS12}, $\varphi$ is an evaluation representation.  In other words, $\varphi$ factors through the Lie algebra homomorphism $\g \otimes A \twoheadrightarrow \g \otimes A/\sm$, where $\sm = (\tau_1,\dotsc,\tau_\ell)$ is the unique maximal ideal of $A$.  In particular, $\g \otimes \sm$ acts by zero on the corresponding module.

  Consider the $\F$-linear symmetrization map
  \[
    \Psi \colon S(\g \otimes A) \to U(\g \otimes A),\quad
    X_1 \dotsm X_n \mapsto \frac{1}{n!} \sum_{\sigma \in \mathfrak{S}_n} X_{\sigma(1)} \dotsm X_{\sigma(n)},
  \]
  where $\mathfrak{S}_n$ denotes the symmetric group on the set $\{1,\dotsc,n\}$ and $U(\g \otimes A)$ denotes the universal enveloping algebra of $\g \otimes A$.  The map $\Psi$ restricts to an isomorphism of vector spaces from $S(\g \otimes A)^{\g \otimes A}$ to the center $Z(\g \otimes A)$ of $U(\g \otimes A)$.
  \details{
    See, for example: D. Calaque and C. A. Rossi. \emph{Lectures on Duflo isomorphisms in Lie algebra and complex geometry}. EMS Series of Lectures in Mathematics. European Mathematical Society (EMS), Z\"urich, 2011.
  }

  Since $\mu > 0$, it follows from Corollary~\ref{cor:inv-poly-gen-set} that $S(\g \otimes A)^{\g \otimes A} \subseteq S(\g \otimes \sm)$.  Thus
  \[
    Z(\g \otimes A) = \Psi \left( S(\g \otimes A)^{\g \otimes A} \right) \subseteq \Psi\big( S(\g \otimes \sm) \big) \subseteq U(\g \otimes \sm).
  \]
  The result follows.
\end{proof}

\begin{rem}
  In spite of Theorem~\ref{theo:center-acts-trivially}, the center of $U(\g \otimes A)$ does not necessarily act by the restriction of the augmentation map on all finite-dimensional modules, since the category of such modules is not semisimple in general.  For example, in the adjoint representation $\fsl_2 \otimes \F[t]/(t^2)$, the center does not act by the restriction of the augmentation map.
\end{rem}


\bibliographystyle{alpha}
\bibliography{Macedo-Savage}

\newcommand{\arxiv}[1]{\href{http://arxiv.org/abs/#1}{\tt
  arXiv:\nolinkurl{#1}}}\newcommand{\doi}[1]{\href{http://dx.doi.org/#1}{\tt
  \nolinkurl{{http://dx.doi.org/#1}}}}
\begin{thebibliography}{FKKS12}

\bibitem[Bag]{Bag15}
I.~Bagci.
\newblock On representations of {C}artan map {L}ie superalgebras.
\newblock \arxiv{1508.00877}.

\bibitem[Bou05]{Bou05}
N.~Bourbaki.
\newblock {\em Lie groups and {L}ie algebras. {C}hapters 7--9}.
\newblock Elements of Mathematics (Berlin). Springer-Verlag, Berlin, 2005.
\newblock Translated from the 1975 and 1982 French originals by Andrew
  Pressley.

\bibitem[CFK10]{CFK10}
V.~Chari, G.~Fourier, and T.~Khandai.
\newblock A categorical approach to {W}eyl modules.
\newblock {\em Transform. Groups}, 15(3):517--549, 2010.
\newblock \doi{10.1007/s00031-010-9090-9}.

\bibitem[CG09]{CG09}
V.~Chari and J.~Greenstein.
\newblock A family of {K}oszul algebras arising from finite-dimensional
  representations of simple {L}ie algebras.
\newblock {\em Adv. Math.}, 220(4):1193--1221, 2009.
\newblock \doi{10.1016/j.aim.2008.11.007}.

\bibitem[CM]{CM16}
L.~Calixto and M.~Macedo.
\newblock Irreducible modules for equivariant map superalgebras and their
  extensions.
\newblock \arxiv{1604.01622}.

\bibitem[CMS16]{CMS16}
L.~Calixto, A.~Moura, and A.~Savage.
\newblock Equivariant map queer {L}ie superalgebras.
\newblock {\em Canad. J. Math.}, 68(2):258--279, 2016.
\newblock \doi{10.4153/CJM-2015-033-6}.

\bibitem[Duf77]{Duf77}
M.~Duflo.
\newblock Op\'erateurs diff\'erentiels bi-invariants sur un groupe de {L}ie.
\newblock {\em Ann. Sci. \'Ecole Norm. Sup. (4)}, 10(2):265--288, 1977.

\bibitem[FKKS12]{FKKS12}
G.~Fourier, T.~Khandai, D.~Kus, and A.~Savage.
\newblock Local {W}eyl modules for equivariant map algebras with free abelian
  group actions.
\newblock {\em J. Algebra}, 350:386--404, 2012.
\newblock \doi{10.1016/j.jalgebra.2011.10.018}.

\bibitem[Geo94]{Geo94a}
F.~Geoffriau.
\newblock Sur le centre de l'alg\`ebre enveloppante d'une alg\`ebre de
  {T}akiff.
\newblock {\em Ann. Math. Blaise Pascal}, 1(2):15--31 (1995), 1994.
\newblock \href{http://www.numdam.org/item?id=AMBP_1994__1_2_15_0}{\tt
  \nolinkurl{http://www.numdam.org/item?id=AMBP_1994__1_2_15_0}}.

\bibitem[Geo95]{Geo95}
F.~Geoffriau.
\newblock Homomorphisme de {H}arish-{C}handra pour les alg\`ebres de {T}akiff
  g\'en\'eralis\'ees.
\newblock {\em J. Algebra}, 171(2):444--456, 1995.
\newblock \doi{10.1006/jabr.1995.1021}.

\bibitem[Hum78]{Hum78}
J.~E. Humphreys.
\newblock {\em Introduction to {L}ie algebras and representation theory},
  volume~9 of {\em Graduate Texts in Mathematics}.
\newblock Springer-Verlag, New York-Berlin, 1978.
\newblock Second printing, revised.

\bibitem[Kam16]{Kam16}
M.~Kamgarpour.
\newblock Compatibility of the {F}eigin-{F}renkel isomorphism and the
  {H}arish-{C}handra isomorphism for jet algebras.
\newblock {\em Trans. Amer. Math. Soc.}, 368(3):2019--2038, 2016.
\newblock \doi{10.1090/S0002-9947-2014-06419-2}.

\bibitem[Kos63]{Kos63}
B.~Kostant.
\newblock Lie group representations on polynomial rings.
\newblock {\em Amer. J. Math.}, 85:327--404, 1963.
\newblock \href{http://www.jstor.org/stable/2373130}{\tt
  \nolinkurl{http://www.jstor.org/stable/2373130}}.

\bibitem[Mus01]{Mus01}
M.~Musta{\c{t}}{\u{a}}.
\newblock Jet schemes of locally complete intersection canonical singularities.
\newblock {\em Invent. Math.}, 145(3):397--424, 2001.
\newblock With an appendix by David Eisenbud and Edward Frenkel.
  \doi{10.1007/s002220100152}.

\bibitem[NS13]{NS13}
E.~Neher and A.~Savage.
\newblock A survey of equivariant map algebras with open problems.
\newblock In {\em Recent developments in algebraic and combinatorial aspects of
  representation theory}, volume 602 of {\em Contemp. Math.}, pages 165--182.
  Amer. Math. Soc., Providence, RI, 2013.
\newblock \doi{10.1090/conm/602/12024}.

\bibitem[NSS12]{NSS12}
E.~Neher, A.~Savage, and P.~Senesi.
\newblock Irreducible finite-dimensional representations of equivariant map
  algebras.
\newblock {\em Trans. Amer. Math. Soc.}, 364(5):2619--2646, 2012.
\newblock \doi{10.1090/S0002-9947-2011-05420-6}.

\bibitem[RT92]{RT92}
M.~Ra{\"{\i}}s and P.~Tauvel.
\newblock Indice et polyn\^omes invariants pour certaines alg\`ebres de {L}ie.
\newblock {\em J. Reine Angew. Math.}, 425:123--140, 1992.
\newblock \doi{10.1515/crll.1992.425.123}.

\bibitem[Sav14]{Sav14}
A.~Savage.
\newblock Equivariant map superalgebras.
\newblock {\em Math. Z.}, 277(1-2):373--399, 2014.
\newblock \doi{10.1007/s00209-013-1261-7}.

\bibitem[Tak71]{Tak71}
S.~J. Takiff.
\newblock Rings of invariant polynomials for a class of {L}ie algebras.
\newblock {\em Trans. Amer. Math. Soc.}, 160:249--262, 1971.

\bibitem[Wil11]{Wil11}
B.~J. Wilson.
\newblock Highest-weight theory for truncated current {L}ie algebras.
\newblock {\em J. Algebra}, 336:1--27, 2011.
\newblock \doi{10.1016/j.jalgebra.2011.04.015}.

\end{thebibliography}

\end{document}